\documentclass{article}

\usepackage[english]{babel}
\usepackage{xcolor}
\usepackage{cite}

\newcommand{\orcid}[1]{\href{https://orcid.org/#1}{\includegraphics[height=0.8em]{orcid.png}}}
\usepackage[letterpaper,top=1.5cm,bottom=1.5cm,left=2cm,right=2cm,marginparwidth=1.75cm]{geometry}

\usepackage{indentfirst}
\usepackage{amssymb}
\usepackage{amsmath}
\usepackage{graphicx}
\usepackage[colorlinks=true, allcolors=blue]{hyperref}

\newtheorem{theorem}{Theorem}
\newtheorem{lemma}[theorem]{Lemma}
\newtheorem{definition}[theorem]{Definition}

\newtheorem{example}[theorem]{Example}
\newtheorem{remark}[theorem]{Remark}

\newtheorem{conjectura}{Conjecture}
\newenvironment{proof}[1][Proof]{\noindent\textbf{#1\,} }{\hfill \rule{0.5em}{0.5em}\medskip}

\title{Interplay between $\mathbb{Z}_2$-gradings and automorphisms in the Grassmann algebra: a survey
}
\author{
	Alan Guimar\~aes\\
	Departamento de Matemática\\
	Universidade Federal do Rio Grande do Norte, Natal, RN, Brazil\\
	alan.guimaraes@ufrn.br }

\begin{document}
	\maketitle
	
	\begin{abstract} 
		
		Let $F$ be a field of characteristic different from two, and let $E$ be the Grassmann algebra of an infinite-dimensional $F$-vector space $L$. In this paper, we survey recent results concerning automorphisms of order two of $E$ and the corresponding induced superalgebra structures. We also present concrete examples of non-homogeneous automorphisms of $E$, as discussed in \cite{agpkauto, anagomes}. The paper concludes with a conjecture regarding the classification of the superalgebra structures of $E$.

		
		\vspace{0.5cm}
		
	\end{abstract}

	\section{Introduction}
	
	Let $F$ be a field of characteristic different from two. Let $L$ be a vector space over $F$ with basis $e_1, e_2, \ldots$. The infinite-dimensional Grassmann algebra $E$ of $L$ over $F$ has a basis $B_E$ consisting of $1$ and all monomials 
	\[
	e_{i_1} e_{i_2} \cdots e_{i_k}
	\]
	where $i_1 < i_2 < \cdots < i_k$, $k \geq 1$. The multiplication in $E$ is induced by the rule
	\[
	e_i e_j = - e_j e_i
	\]
	for all $i$ and $j$; in particular, $e_i^2 = 0$ for each $i$.
	
	The Grassmann algebra is the most natural example of a superalgebra and is widely used in various areas of Mathematics as well as in Theoretical Physics. Moreover, $E$ is one of the most important polynomial identity (PI) algebras. Its polynomial identities were first described by Latyshev \cite{latyshev} and later studied in detail by Krakowski and Regev \cite{KR}. It was observed by P.~M.~Cohn that the Grassmann algebra satisfies no standard identities. Recall that the standard polynomial $s_n$ is the alternating sum of all monomials obtained by permuting the variables in $x_1 x_2 \cdots x_n$. The true significance of the Grassmann algebra in PI theory was revealed through the work of A.~Kemer in the mid-1980s. Kemer proved that every associative PI-algebra over a field of characteristic zero is PI-equivalent to the Grassmann envelope of a finite-dimensional associative superalgebra; see \cite{Kemer1, Kemer2}.
	
	The description of gradings on an algebra by a group is an important problem in the structure theory of graded rings and its applications. The gradings on ample classes of algebras have already been classified; these include the full matrix algebras, the upper triangular matrix algebras and block-triangular matrix algebras. The gradings by an arbitrary group for the algebras of upper triangular matrices over an arbitrary field were classified by A. Valenti and M. Zaicev, see \cite{VZ}. A complete classification of the group gradings on full matrix algebras and block-triangular matrix algebras has been given under some restrictions on the field, which is assumed to be algebraically closed, and on the grading group, see \cite{BZ}, \cite{DGK}. In \cite{AGDD}, the authors classified the gradings on matrix algebras of prime order, under mild assumptions on the ground field.

	The graded polynomial identities satisfied by an algebra are often ``easier" to describe than the ordinary ones. Indeed, polynomial identities of associative algebras are known in very few cases, including the Grassmann algebra $E$ (over any field), the $2 \times 2$ matrix algebra \cite{second-order D, Yu Raz} (in characteristic 0), \cite{plamen order 2} (over infinite fields of characteristic different from 
	2), and \cite{malcev kuzmin} (over a finite field); the upper triangular matrix algebras \cite{malcev only, chiripov}, and the algebra $E\otimes E$ \cite{popov} in 
	characteristic 0. On the other hand, gradings on matrix algebras are well understood \cite{Bah}, and their graded identities are thoroughly describe. Gradings and graded identities on upper triangular matrix algebras have also been classified \cite{VPV, VZ}. Furthermore, graded identities for natural gradings on other important classes of algebras have been studied extensively.
	
	The Grassmann algebra admits a natural grading by the cyclic group $\mathbb{Z}_2$. Its structure from the viewpoint of PI-theory is well known and straightforward to deduce; see, for example, \cite{GMZ}. In recent years, many papers have presented results on gradings and their graded identities for the Grassmann algebra. In \cite{disil}, the authors studied all homogeneous superalgebra structures defined on the Grassmann algebra. Their $\mathbb{Z}_2$-graded polynomial identities were also described in \cite{disil, centroneP, LFG}.
	
	A common approach to studying $\mathbb{Z}_2$-gradings of an algebra $A$ is to investigate automorphisms of order two of $A$. If $\varphi \in \operatorname{Aut}(A)$ is an automorphism of order at most two, i.e., $\varphi^2 = \mathrm{id}$, then $A$ decomposes as a $\mathbb{Z}_2$-grading
	\[
	A = A_{0,\varphi} \oplus A_{1,\varphi},
	\]
	where $A_{0,\varphi}$ and $A_{1,\varphi}$ are the eigenspaces of $\varphi$ corresponding to eigenvalues $1$ and $-1$, respectively. Conversely, to each $\mathbb{Z}_2$-grading on $A$ one can associate an automorphism of order at most two defined by
	\[
	\varphi(a_0 + a_1) = a_0 - a_1
	\]
	for every $a_i \in A_i$, $i=0,1$. This duality between group gradings and group actions is well documented; see, for example, \cite{GMZ} for a general discussion. It is worth noting that, in general, the description of automorphisms of algebras constitutes a challenging problem in ring theory. In \cite{makar}, Makar-Limanov provided a description of the group of automorphisms of finite-dimensional Grassmann algebras. In \cite{AK}, the author studies an algorithm for computing automorphisms of the infinite-dimensional Grassmann algebra. We point out that, to this day, no complete description of the group of automorphisms of the infinite-dimensional Grassmann algebra is known. In \cite{agpkauto, anagomes}, the authors study superalgebra structures on $E$ by means of concrete automorphisms of order two acting on $E$. They investigate polynomial identities, PI-equivalence, and the isomorphism problem for superalgebras of the Grassmann algebra. In particular, the reference \cite{anagomes} provides a characterization of the natural $\mathbb{Z}_2$-grading of $E$ in terms of its $\mathbb{Z}_2$-graded polynomial identities.
	
	Beyond the investigation of the $\mathbb{Z}_2$-gradings on $E$, numerous works have explored gradings on $E$ by other groups. We cite here the papers \cite{centroneG, disilplamen, APzgrad, AGuim, ABC, ACP}. The $\ast$-identities of any involution algebra structure on $E$ were described in \cite{CDD}. The differential identities of 
	$E$ with respect to the action of a finite-dimensional Lie algebra  of inner derivations were studied, and a set of generators for the ideal of differential identities was computed in \cite{Rizzo}. This highlights that the Grassmann algebra has also been investigated in the non-associative setting.

	In this paper, we survey recent results on $\mathbb{Z}_2$-gradings and automorphisms of order two of the Grassmann algebra, as presented in \cite{agpkauto, anagomes}. In Section~\ref{prmilinaries}, we introduce the definitions and notation necessary for understanding the paper. In Section~\ref{main}, we present strategies for constructing examples of automorphisms of the Grassmann algebra and the corresponding $\mathbb{Z}_2$-gradings on $E$. We conclude by stating a conjecture related to the classification of these structures.

	\section{Preliminaries}\label{prmilinaries}
	
	In this section, we present the concepts, notation, and preliminary results that will be essential for understanding the remainder of the paper.
	
	\subsection{$\mathbb{Z}_2$-gradings and identities}
	
	\begin{definition}
		Let $F$ be a field and $A$ be a unital associative $F$-algebra. We say that $A$ is a \emph{$\mathbb{Z}_2$-graded algebra} (or \emph{superalgebra}) if
		\[
		A = A_0 \oplus A_1
		\]
		where $A_0, A_1$ are $F$-subspaces such that
		\[
		A_i A_j \subseteq A_{i+j}
		\]
		for all $i, j \in \mathbb{Z}_2$. The subspace $A_0$ is called the \emph{even component} and $A_1$ the \emph{odd component} of the grading. A non-zero vector $a \in A_i$ is called \emph{homogeneous of degree} $i$, and we denote this by $\|a\| = i$. A vector subspace (subalgebra, ideal) $W \subseteq A$ is said to be \emph{$\mathbb{Z}_2$-graded} if
		\[
		W = (W \cap A_0) \oplus (W \cap A_1).
		\]
	\end{definition}
	
	\begin{example}
		Every algebra $A$ admits the \emph{trivial} $\mathbb{Z}_2$-grading
		\[
		A = A_0 \oplus A_1,
		\]
		where $A_0 = A$ and $A_1 = \{0\}$.
	\end{example}
	
	\begin{example}
		The matrix algebra $M_2(F)$ admits the $\mathbb{Z}_2$-grading
		\[
		M_2(F) = A_0 \oplus A_1,
		\]
		where
		\[
		A_0 = \mathrm{span}\{e_{11}, e_{22}\} \quad \text{and} \quad A_1 = \mathrm{span}\{e_{12}, e_{21}\}.
		\]
	\end{example}
	
	\begin{example}
		In the Grassmann algebra $E$, consider the subspaces $E_{(0)}$, defined as the span of $1$ and all monomials of even length, and $E_{(1)}$, defined as the span of all monomials of odd length.

		Then
		\[
		E = E_{(0)} \oplus E_{(1)},
		\]
		and it satisfies
		\[
		E_{(0)} E_{(0)} + E_{(1)} E_{(1)} \subseteq E_{(0)}, \quad E_{(0)} E_{(1)} + E_{(1)} E_{(0)} \subseteq E_{(1)}.
		\]
		Hence, $E = E_{(0)} \oplus E_{(1)}$ is a $\mathbb{Z}_2$-grading, which we denote by $E_{can}$. In this grading, we have that $e_i$ is homogeneous of degree $1$, for all $i$.
	\end{example}
	
	\medskip
	
	If $A$ and $B$ are superalgebras, a homomorphism $f : A \to B$ is \emph{$\mathbb{Z}_2$-graded} if
	\[
	f(A_i) \subseteq B_i
	\]
	for all $i \in \mathbb{Z}_2$. When there exists a bijective $\mathbb{Z}_2$-graded homomorphism between $A$ and $B$, we say that $A$ and $B$ are \emph{$\mathbb{Z}_2$-isomorphic}.
	
	\medskip
	
	\noindent
	One defines a free object in the category of superalgebras by considering the free associative $F$-algebra generated by the disjoint union of two countable sets of variables, denoted by $Y$ and $Z$. We assume that the elements of $Y$ are of degree zero and the elements of $Z$ are of degree one. This algebra is denoted by
	\[
	F\langle Y \cup Z \rangle.
	\]
	Its even part is the vector space spanned by all monomials containing an even number of variables from $Z$, and the odd part is spanned by all monomials with an odd number of variables from $Z$. It is straightforward to verify that $F\langle Y \cup Z \rangle$ is the free associative superalgebra on the graded set $Y \cup Z$, in the sense that for any superalgebra $A$ and any map
	\[
	\varphi : Y \cup Z \to A
	\]
	satisfying $\varphi(Y) \subseteq A_0$ and $\varphi(Z) \subseteq A_1$, there exists a unique $\mathbb{Z}_2$-graded homomorphism
	\[
	F\langle Y \cup Z \rangle \to A
	\]
	extending $\varphi$. For $x \in Y \cup Z$, we denote by $\|x\|$ its $\mathbb{Z}_2$-degree. Given a subset $S \subseteq F\langle Y \cup Z \rangle$, the \emph{$T_2$-ideal} generated by $S$ is denoted by $\langle S \rangle_{T_2}$.
	
	\medskip
	
	\noindent
	A polynomial
	\[
	f(y_1, \ldots, y_l, z_1, \ldots, z_m) \in F\langle Y \cup Z \rangle
	\]
	is a \emph{$\mathbb{Z}_2$-graded polynomial identity} of a superalgebra $A$ if
	\[
	f(a_1, \ldots, a_l, b_1, \ldots, b_m) = 0
	\]
	for all substitutions with $\|a_i\| = 0$ and $\|b_j\| = 1$. The set of all $\mathbb{Z}_2$-graded polynomial identities of $A$,
	denoted by $T_2(A)$, is a homogeneous ideal of $F\langle Y \cup Z \rangle$ called the \emph{$T_2$-ideal} of $A$.

	\subsection{$\mathbb{Z}_2$-gradings and automorphisms: a duality}
	
	Let $A$ be an associative $F$-algebra. There exists a natural duality between $\mathbb{Z}_2$-gradings and automorphisms of order 
	$\leq 2$ on $A$. The duality is defined as follows.
	
	If $\varphi \in \operatorname{Aut}(A)$ is such that $\varphi^2 = 1$ then $A_{\varphi} = A_{0,\varphi} \oplus A_{1,\varphi}$ where the homogeneous 
	components are the eigenspaces corresponding to the eigenvalues $1$ and $-1$ of $\varphi$, respectively. The decomposition in a direct sum 
	of the eigenspaces exists since $F$ is of characteristic different from $2$.
	
	The general facts about duality between gradings and actions of groups can be found, for example, in \cite{GMZ}. 
	We observe that the homogeneous $\mathbb{Z}_2$-gradings on $E$ correspond to the automorphisms on $E$ satisfying
	\[
	\varphi(e_i) = \pm e_i.
	\]
	
	If $\varphi \in \operatorname{Aut}(E)$ with $\varphi^2 = id$ we observe that
	\[
	e_i = \frac{e_i + \varphi(e_i)}{2} + \frac{e_i - \varphi(e_i)}{2}, \quad \text{for } i \in \mathbb{N}.
	\]
	Setting $a_i = e_i + \varphi(e_i)/2$ we have:
	\begin{itemize}
		\item $\varphi(e_i) = -e_i + 2a_i$,
		\item $\varphi(a_i) = a_i$, that is, $a_i$ is of degree zero in the $\mathbb{Z}_2$-grading $E_\varphi$,
		\item $\varphi(e_i - a_i) = -(e_i - a_i)$, that is, $e_i - a_i$ is of degree $1$ in the $\mathbb{Z}_2$-grading $E_\varphi$.
	\end{itemize}
	
	\begin{definition}
		Let $\varphi \in \operatorname{Aut}(E)$. We say that $\varphi$ is of \emph{canonical type} if $\varphi(e_{i})\in E_{(1)}$, for all $i$.
	\end{definition}

	If $\varphi$ is an automorphism of order $2$ on $E$, we have that $\varphi$ is of canonical type if and only if $a_i \in E_{(1)}$ for all $i$.

	Let us fix a basis $\beta = \{e_1, e_2, \ldots, e_n, \ldots\}$ of the vector space $L$ and an automorphism 
	$\varphi \in \operatorname{Aut}(E)$ such that $\varphi^2 = id$. Then $\varphi$, as a linear transformation, has eigenvalues $1$ and $-1$ only, 
	and moreover, there exists a basis of the vector space $E$ consisting of eigenvectors. (It is well known from elementary Linear Algebra 
	that this fact does not depend on the dimension of the vector space as long as the characteristic of $F$ is different from $2$.) Then 
	$E = E(1) \oplus E(-1)$ where $E(t)$ is the eigenspace for the eigenvalue $t$ of the linear transformation $\varphi$. One considers the 
	intersections $L(t) = L \cap E(t)$, $t = \pm 1$. Changing the basis $\beta$, if necessary, one may assume that $L(t)$ is the span of 
	$\beta \cap L(t)$. Clearly this change of basis gives rise to a homogeneous automorphism of $E$ and we can take the composition of it 
	and then $\varphi$. We shall assume that such a change of basis has been done.
	
	Denote 
	\[
	I_\beta = \{ n \in \mathbb{N} \mid \varphi(e_n) = \pm e_n \}.
	\]
	We shall distinguish the following four possibilities:
	\begin{enumerate}
		\item $I_\beta = \mathbb{N}$.
		\item $I_\beta \neq \mathbb{N}$ is infinite.
		\item $I_\beta$ is finite and nonempty.
		\item $I_\gamma = \emptyset$ for every linear basis $\gamma$ of $L$.
	\end{enumerate}
	
	We shall call these automorphisms (and also the corresponding $\mathbb{Z}_2$-gradings), automorphisms (or gradings) of type $1$, $2$, $3$, 
	and $4$, respectively.
	
	The automorphisms of type 1 induce $\mathbb{Z}_2$-gradings on $E$ in which all generators of $E$ are homogeneous. Such structures are called homogeneous $\mathbb{Z}_2$-gradings on $E$. The corresponding graded identities were completely studied in \cite{disil, centroneP, LFG}.

	We conclude this section with the following lemma.
	
	\begin{lemma}\label{v=0}
		Let $\varphi$ be an automorphism of order two of $E$. Then $\varphi$ is of type~4 if and only if, for every $v \in L$ such that $\varphi(v) = \pm v$, one has $v = 0$.
	\end{lemma}
	\begin{proof}
		Assume that $\varphi$ is of type~ 4 and let $v \in L$ with $\varphi(v) = \pm v$. 
		If $v \neq 0$, choose a basis $\gamma$ of $L$ such that $v \in \gamma$. 
		Then $I_{\gamma} \neq \emptyset$, a contradiction. 
		The converse follows by the same argument.
	\end{proof}

	\section{Automorphisms of order two of $E$}\label{main}
	
	From this point on, our goal is to survey recent developments regarding automorphisms of order two and the corresponding $\mathbb{Z}_2$-gradings of the infinite-dimensional Grassmann algebra.
	
	Let $X=\{e_{1}, \ldots, e_{n},\ldots\}$. For each map $\lambda : X\to E$, we can define the linear transformation $\varphi: E\to E$ by
	\begin{equation}\label{estend to a linear trans}
		\varphi(e_{i_1}\ldots e_{i_n})=\lambda(e_{i_1})\ldots \lambda(e_{i_n}),
	\end{equation}
	for all $n\in\mathbb{N}$.
	
	We start with the next lemma.
	
	\begin{lemma}\label{linear trans for end}
		The linear transformation $\varphi$ is an endomorphism of $E$ if and only if 
		\[
		\lambda(e_{i})\lambda(e_{j}) + \lambda(e_{j})\lambda(e_{i}) = 0, \quad \text{for all } i, j.
		\]
	\end{lemma}
	
	\begin{proof}
		We only need to prove the converse.  
		Let $x = e_{i_1}\cdots e_{i_m}$ and $y = e_{j_1}\cdots e_{j_n}$ be monomials in $B_E$. If $xy = 0$, it is clear that $\varphi(xy) = 0 = \varphi(x)\varphi(y)$.  
		
		Now assume $xy \neq 0$. Then there exists a permutation $\sigma$ such that
		\[
		\sigma(i_1) < \cdots < \sigma(i_m) < \sigma(j_1) < \cdots < \sigma(j_n).
		\] 
		In this case, we have
		\[
		\varphi(xy) 
		= \varphi\!\left((-1)^{\sigma} e_{\sigma(i_1)} \cdots e_{\sigma(i_m)} e_{\sigma(j_1)} \cdots e_{\sigma(j_n)}\right) 
		= (-1)^{\sigma} e_{\sigma(i_1)} \cdots e_{\sigma(i_m)} e_{\sigma(j_1)} \cdots e_{\sigma(j_n)},
		\]
		where $(-1)^{\sigma}$ denotes the sign of $\sigma$.  
		
		On the other hand,
		\[
		\varphi(x)\varphi(y) 
		= \big(\lambda(e_{i_1}) \cdots \lambda(e_{i_m})\big)\big(\lambda(e_{j_1}) \cdots \lambda(e_{j_n})\big) 
		= (-1)^{\sigma} e_{\sigma(i_1)} \cdots e_{\sigma(i_m)} e_{\sigma(j_1)} \cdots e_{\sigma(j_n)}.
		\]
		
		Since $\varphi$ is linear, the result follows.
	\end{proof}

	Therefore, it follows from Lemma~\ref{linear trans for end} that, to construct an automorphism 
	$\varphi$ of $E$ of order two, we must define $\varphi$ so that
	\[
	\varphi(e_{i})\varphi(e_{j}) + \varphi(e_{j})\varphi(e_{i}) = 0\]
	and 
	\[\varphi^{2}(e_{i}) = e_{i},
	\]
	for all $i$ and $j$. From this point onward, we will adopt this approach.

	\subsection{Concrete automorphisms of type 1}
	
	To define a homogeneous $\mathbb{Z}_{2}$-grading on $E$, it is sufficient to assign degrees to the elements of a basis of $L$. More specifically, according to \cite{disil}, for $k \in \mathbb{N}_0 = \mathbb{N} \cup \{0\}$, the possible assignments are:
	\[
	\|e_{i}\|_{k}=
	\begin{cases} 
		0, & \text{if } i=1,\ldots,k, \\[2pt] 
		1, & \text{otherwise},
	\end{cases}
	\]
	\[
	\|e_{i}\|_{k^\ast}=
	\begin{cases} 
		1, & \text{if } i=1,\ldots,k, \\[2pt] 
		0, & \text{otherwise},
	\end{cases}
	\]
	and
	\[
	\|e_{i}\|_{\infty}=
	\begin{cases} 
		0, & \text{if $i$ is even}, \\[2pt] 
		1, & \text{otherwise}.
	\end{cases}
	\]
	The degree of a monomial $e_{i_1}e_{i_2}\cdots e_{i_t}$ is given by
	\[
	\|e_{i_1}e_{i_2}\cdots e_{i_t}\| 
	= \|e_{i_1}\| + \|e_{i_2}\| + \cdots + \|e_{i_t}\|, 
	\]
	where the sum is taken in $\mathbb{Z}_{2}$. These gradings are denoted by $E_{k}$, $E_{k^\ast}$, and $E_{\infty}$, respectively.  
	
	When $\|e_{i}\| = 1$ for all $i$, we recover the canonical structure $E_{can}$. Hence, up to isomorphism, these are precisely all the homogeneous superalgebra structures on $E$.

	Moreover, every homogeneous $\mathbb{Z}_{2}$-grading on $E$ corresponds to an automorphism $\varphi$ of $E$ satisfying, for a suitable basis of $E$,
	\[
	\varphi(e_i) = \pm e_i.
	\]
	In particular, the automorphism $\varphi(e_{i}) = -e_{i}$ for all $i$ induces the canonical superalgebra $E_{can}$, while $\varphi(e_{i}) = e_{i}$ induces the trivial superalgebra of $E$. We recall that the $\mathbb{Z}_2$-graded identities for all homogeneous superalgebras of $E$ were described in \cite{disil, centroneP, LFG}.

	\subsection{Concrete automorphisms of type 2}

	In this subsection we construct certain automorphisms of type~2. We start with the next remark.

	\begin{remark}
		In \cite{makar}, the author provides a classification of the automorphism group of finite-dimensional Grassmann algebras. 
		A natural idea to obtain automorphisms of type~2 of $E$ is the following. 
		Let $\varphi$ be an automorphism of the finite-dimensional Grassmann algebra generated by $e_{1},\ldots,e_{n}$, for some $n$. 
		We then define a map $\psi \colon \{e_{1}, e_{2}, \ldots\} \to E$ by
		\[
		\psi(e_{i}) =
		\begin{cases}
			\varphi(e_{i}), & \text{if } i \leq n, \\[4pt]
			\pm e_{i}, & \text{if } i > n .
		\end{cases}
		\]
		
		We extend $\psi$ linearly to $E$. 
		It is straightforward to check that $\psi$ is an automorphism of $E$ if and only if $\varphi(e_{i})$ is a linear combination of monomials of odd length for all $1 \leq i \leq n$. 
		Moreover, in this case, if $\varphi$ has order two, then so does $\psi$. 
		Next, we provide a more general approach to obtain automorphisms of type~2.
	\end{remark}

	Let $I \subset \mathbb{N}$ be an infinite set such that $I \neq \mathbb{N}$, 
	and define the map $\varphi(e_i) = \pm e_i$ for each $i \in I$. 
	We then set
	\[
	I^+ = \{i \in I \mid \varphi(e_i) = e_i\}, \quad 
	I^- = \{i \in I \mid \varphi(e_i) = -e_i\}, \quad
	J = \{j \in \mathbb{N} \mid j \notin I\}.
	\]
	Note that $I = I^+ \cup I^-$ is a disjoint union.  
	For each $j \in J$, we construct elements $d_j \in E$ satisfying:
	\begin{enumerate}
		\item[(1)] $d_j \in E_{(1)}$;
		\item[(2)] every monomial in $d_j$ is a product of generators whose indices lie in $I$;
		\item[(3)] every monomial in $d_j$ contains an even number of factors from $I^-$.
	\end{enumerate}
	
	We extend the action of $\varphi$ to all $e_i$ by:
	\[
	\varphi(e_i) = 
	\begin{cases}
		e_i, & \text{if } i \in I^+, \\[2pt]
		-e_i, & \text{if } i \in I^-, \\[2pt]
		-e_i + d_i, & \text{if } i \in J.
	\end{cases}
	\]
	As in Equality (\ref{estend to a linear trans}), we extend $\varphi$ to a linear transformation of $E$.
	
	The following observations hold:
	\begin{itemize}
		\item Condition~(1) ensures that 
		$\varphi(e_{i})\varphi(e_{j}) + \varphi(e_{j})\varphi(e_{i}) = 0$ 
		for all $i, j$. By Lemma~\ref{linear trans for end}, it follows that $\varphi$ is an endomorphism of $E$.
		\item Conditions~(2) and~(3) imply that $\varphi(d_j) = d_j$ for every $j \in J$.
	\end{itemize}
	
	We claim that $\varphi$ is an automorphism of order at most $2$. Indeed, 
	if $i \in I$ we have $\varphi^2(e_i) = e_i$.  
	If $j \in J$, then
	\[
	\varphi^2(e_j) = \varphi(-e_j + d_j) 
	= -(-e_j + d_j) + d_j 
	= e_j.
	\]
	Thus, $\varphi$ is an automorphism of $E$ of order $\leq 2$.  
	Let $E_\varphi = E_{0,\varphi} \oplus E_{1,\varphi}$ be the superalgebra determined by $\varphi$.  
	Its homogeneous components satisfy
	
	\[E_{0,\varphi} \supset \operatorname{span}_F\{e_i, d_j \mid i \in I^+,\, j \in J\}\] 
	and
	\[E_{1,\varphi} \supset \operatorname{span}_F\{e_i, e_j - d_j \mid i \in I^-,\, j \in J\}.\]

	\begin{example}
		Let $I = I^+ = \{2, 3, 4, \ldots\}$ and define $\varphi$ by:
		\[
		\varphi(e_i) =
		\begin{cases}
			e_i, & \text{if } i \neq 1, \\[2pt]
			-e_1 + e_2e_3e_4, & \text{if } i = 1.
		\end{cases}
		\]
		Then $\varphi$ is an automorphism of type~2.
	\end{example}
	
	We refer to the above procedure for obtaining an automorphism on $E$ as \textbf{method~A}.  
	Using method~A, we can construct numerous examples of type~2 automorphisms of the Grassmann algebra.

	\subsection{Concrete automorphisms of type~3}
	
	We now present a method for constructing automorphisms of $E$ of type~3.  
	We will see that such automorphisms need not necessarily be of canonical type.
	
	Let $t$ and $k$ be integers such that $t$ is odd and $k \geq 0$.  
	Set
	\[
	I = \{1, \ldots, k, k+1, \ldots, k+t\}.
	\]
	For each $i \in I$, define $\varphi(e_{i}) = \pm e_{i}$ so that
	\[
	I^{+} = \{1, \ldots, k\}, 
	\quad 
	I^{-} = \{k+1, \ldots, k+t\}.
	\]
	For $n > k+t$, set
	\[
	\varphi(e_{n}) = -e_{n} + \lambda_{n}e_{1} \cdots e_{k} \, e_{k+1} \cdots e_{k+t} \, e_{n},
	\]
	where $\lambda_{n}\neq 0$. Thus, the action of $\varphi$ on the generators is
	\[
	\varphi(e_{n}) =
	\begin{cases}
		e_{n}, & \text{if } 1 \leq n \leq k, \\[2pt]
		-e_{n}, & \text{if } k+1 \leq n \leq k+t, \\[2pt]
		-e_{n} + \lambda_{n}e_{1} \cdots e_{k} e_{k+1} \cdots e_{k+t} e_{n}, & \text{if } n > k+t.
	\end{cases}
	\]
	
	We observe:
	\begin{itemize}
		\item For all $i,j$, we have $\varphi(e_{i})\varphi(e_{j}) + \varphi(e_{j})\varphi(e_{i}) = 0$.  
		By Lemma~\ref{linear trans for end}, it follows that $\varphi$ is an endomorphism of $E$.
		\item If $n > k+t$, then
		\begin{align*}
			\varphi^{2}(e_{n}) 
			&= \varphi\big(-e_{n} + \lambda_{n}e_{1} \cdots e_{k} e_{k+1} \cdots e_{k+t} e_{n}\big) \\
			&= -\big(-e_{n} + \lambda_{n}e_{1} \cdots e_{k} e_{k+1} \cdots e_{k+t} e_{n}\big) \\
			&\quad + \lambda_{n}e_{1} \cdots e_{k} (-1)^{t} e_{k+1} \cdots e_{k+t} \big(-e_{n} + \lambda_{n}e_{1} \cdots e_{k} e_{k+1} \cdots e_{k+t} e_{n}\big) \\
			&= e_{n}.
		\end{align*}
	\end{itemize}
	
	Hence $\varphi$ is an automorphism of $E$ of order $2$.  
	The induced superalgebra $E_{\varphi} = E_{0,\varphi} \oplus E_{1,\varphi}$ has homogeneous components satisfying
	
	\[E_{0,\varphi} \supset \operatorname{span}_{F}\{e_{1}, \ldots, e_{k}\}\]
	and
	\[E_{1,\varphi} \supset \operatorname{span}_{F}\{e_{k+1}, \ldots, e_{k+t}, \; e_{n} - e_{1} \cdots e_{k} e_{k+1} \cdots e_{k+t} e_{n} \mid n > k+t\}.\]

	We refer to this procedure as \textbf{method~B} for constructing Grassmann superalgebras.  
	When $k$ is even, the automorphism $\varphi$ is not of canonical type.  
	Nevertheless, as we shall see below, the resulting superalgebras are isomorphic to homogeneous ones.  
	
	\begin{example}
		\label{prop_minus}
		Let $\varphi \colon E \to E$ be defined by:
		\[
		\varphi(e_{n}) =
		\begin{cases}
			-e_{1}, & \text{if } n = 1, \\[2pt]
			-e_{n} + 2e_{1}e_{n}, & \text{if } n > 1.
		\end{cases}
		\]
		Then $\varphi$ is an automorphism of order~2.  
		Let $E_{\varphi} = E_{0,\varphi} \oplus E_{1,\varphi}$ be the induced superalgebra.  
		Its homogeneous components are:
		
		\[E_{1,\varphi} \supset \operatorname{span}_{F}\{e_{1}, \; e_{n} - e_{1}e_{n} \mid n > 1\}\]
		and
		\[E_{0,\varphi} = Z(E).\]
		In this case, only one generator of $E$, namely $e_{1}$, is homogeneous.  
		Moreover, $\varphi$ is not of canonical type.  
		However, it is possible to show that $E_{\varphi}$ is isomorphic to the natural $\mathbb{Z}_{2}$-grading $E_{can}$.
	\end{example}

	\subsection{Concrete automorphisms of type 4}

	In order to provide an example of an automorphism of type 4, we start with the following definition.
	
	\begin{definition}\label{def epsilon}
		Let
		\[
		I = \left\{ 2n \in \mathbb{N},\ n > 1 \ \middle|\ \exists\, m \in \mathbb{N} \ \text{such that} \ 2^{m} \leq n < 2^{m} + 2^{m-1} \right\}.
		\]
		We define the sequence $\{\epsilon_i\}_{i \in \mathbb{N}}$ by
		\[
		\begin{aligned}
			&\epsilon_{2n-1} = -1, && \text{if } 2n \in I,\\
			&\epsilon_{2n-1} = 1, && \text{if } 2n \in \mathbb{N} \setminus I,\\
			&\epsilon_{2n} = 1, && \text{for all } n \in \mathbb{N}.
		\end{aligned}
		\]
	\end{definition}
	
	\begin{remark}\label{ex epsilon}
		According to Definition~\ref{def epsilon}, we have:
		\[
		\epsilon_{1} = 1,\quad \epsilon_{2} = 1,\quad \epsilon_{3} = -1,\quad \epsilon_{4} = 1,\quad \epsilon_{5} = 1,\quad \epsilon_{6} = 1,\quad \epsilon_{7} = -1, \ \text{and so on}.
		\]
	\end{remark}

	\begin{lemma}[\cite{anagomes}, Lemma~13]\label{lemma epsilon}
		For every natural number $n$, we have
		\[
		\epsilon_{1} \cdots \epsilon_{2n+1} \;=\; -\epsilon_{n}.
		\]
	\end{lemma}
	
	\begin{proof}
		See \cite[ Lemma~13]{anagomes} for details.
	\end{proof}

	Let $w_{n} = e_{1}e_{2} \cdots e_{2n+1}$.  Next, we construct an automorphism of type 4 of the Grassmann algebra $E$.  
	We extend the map $\varphi \colon X \to E$ defined by
	\[
	\varphi(e_{i}) = \epsilon_{i} e_{i} + w_{i},
	\]
	to a linear transformation of $E$. For simplicity, we continue to denote this extension by $\varphi$.

	The following theorem is the main result of this section.
	
	\begin{theorem}[\cite{anagomes}, Theorem 14]
		The linear transformation $\varphi$ defined above is an automorphism of type 4 of the Grassmann algebra.
	\end{theorem}
	
	\begin{proof}
		From Lemma \ref{linear trans for end} any linear transformation $\phi$ on $E$ satisfying
		\[
		\phi(e_i)\phi(e_j) + \phi(e_j)\phi(e_i) = 0,
		\]
		for all $i, j \in \mathbb{N}$, can be extended to a unique endomorphism of $E$.  
		Since $\epsilon_{i} e_{i} + w_{i} \in E_{(1)}$, it follows that $\varphi$ can be extended to an endomorphism of $E$.  
		
		Moreover,
		\[
		\varphi^{2}(e_{i}) = \epsilon_{i}(\epsilon_{i} e_{i} + w_{i}) + \varphi(w_{i}).
		\]
		Notice that
		\[
		\varphi(w_{i}) = \varphi(e_{1}) \cdots \varphi(e_{2i+1})
		= \epsilon_{1} \cdots \epsilon_{2i+1} w_{i}
		= -\epsilon_{i} w_{i},
		\]
		where the last equality follows from Lemma~\ref{lemma epsilon}.  
		Thus, $\varphi^{2}(e_{i}) = e_{i}$, and hence $\varphi$ is an automorphism of order $2$.
		
		Finally, we show that $\varphi$ is of type 4.  
		Assume $v \in L$ satisfies $\varphi(v) = \pm v$.  
		Then there exist $\alpha_{1}, \ldots, \alpha_{n} \in F$ such that
		\[
		v = \alpha_{1} e_{i_1} + \cdots + \alpha_{n} e_{i_n}.
		\]
		Applying $\varphi$, we obtain
		\[
		\pm \left( \sum_{k=1}^{n} \alpha_{k} e_{i_k} \right)
		= \sum_{k=1}^{n} \alpha_{k} \left( \epsilon_{i_k} e_{i_k} + w_{i_k} \right).
		\]
		We deduce
		\[
		\sum_{k=1}^{n} \alpha_{k} w_{i_k} = 0.
		\]
		Since the set $\{ w_i \mid i \in \mathbb{N} \}$ is linearly independent, it follows that $\alpha_{1} = \cdots = \alpha_{n} = 0$.  
		From Lemma \ref{v=0}, we conclude that $\varphi$ is of type 4.
	\end{proof}
	
	We call the procedure described above for producing automorphisms of type 4 \textbf{method~C}.

	\subsection{A theorem and a conjecture}

	We conclude this paper by presenting a theorem and a conjecture concerning the superalgebra structures of the Grassmann algebra.
	
	\begin{theorem}
		If $\varphi$ is an automorphism constructed by Method~A, Method~B, or Method~C, then there exists an automorphism $\psi$ of type~1 such that $E_{\varphi}$ and $E_{\psi}$ are isomorphic superalgebras.
	\end{theorem}
	\begin{proof}
		The claim follows directly from Propositions~7 and~10 in \cite{agpkauto}, together with Proposition~15 in \cite{anagomes}.
	\end{proof}

	The preceding theorem naturally leads to the following conjecture.
	
	\begin{conjectura}
		Every $\mathbb{Z}_{2}$-grading of $E$ is $\mathbb{Z}_{2}$-graded isomorphic to some homogeneous $\mathbb{Z}_{2}$-grading of $E$.
	\end{conjectura}
	
	This conjecture was first stated in \cite{anagomes} and remains open to this day. If it is true, it would provide a classification of the superalgebra structures on $E$. Otherwise, it would imply the existence of some interesting superalgebra structures on $E$ that are not yet known.

\end{document}